\documentclass[leqno,english]{amsart}
\usepackage{helvet}

\usepackage[T1]{fontenc}
\usepackage[latin9]{inputenc}
\usepackage{color}
\usepackage{babel}
\usepackage{enumitem}
\usepackage{amstext}
\usepackage{amsthm}
\usepackage{amssymb}
\usepackage[unicode=true,pdfusetitle,
 bookmarks=true,bookmarksnumbered=true,bookmarksopen=false,
 breaklinks=false,pdfborder={0 0 0},pdfborderstyle={},backref=false,colorlinks=true]
 {hyperref}
\hypersetup{
 citecolor=blue,linkcolor=red}

\makeatletter
\numberwithin{equation}{section}
\numberwithin{figure}{section}
\numberwithin{table}{section}
\theoremstyle{plain}
\newtheorem{thm}{\protect\theoremname}[section]
\theoremstyle{plain}
\newtheorem{conjecture}[thm]{\protect\conjecturename}
\theoremstyle{plain}
\newtheorem{prop}[thm]{\protect\propositionname}
\theoremstyle{definition}
\newtheorem{defn}[thm]{\protect\definitionname}
\theoremstyle{plain}
\newtheorem{lem}[thm]{\protect\lemmaname}
\newlist{casenv}{enumerate}{4}
\setlist[casenv]{leftmargin=*,align=left,widest={iiii}}
\setlist[casenv,1]{label={{\itshape\ \casename} \arabic*.},ref=\arabic*}
\setlist[casenv,2]{label={{\itshape\ \casename} \roman*.},ref=\roman*}
\setlist[casenv,3]{label={{\itshape\ \casename\ \alph*.}},ref=\alph*}
\setlist[casenv,4]{label={{\itshape\ \casename} \arabic*.},ref=\arabic*}

\usepackage{tikz,fancyhdr,enumitem}
\usetikzlibrary{scopes,calc,intersections,through,backgrounds,arrows.meta,decorations.pathmorphing,positioning,fit,petri,decorations.markings,matrix}
\usetikzlibrary{decorations.pathreplacing,calligraphy}
\usepackage{hyperref,bookmark,dsfont,nicefrac,bbm}
\usepackage[all]{xy}
\hypersetup{linktocpage}

\usepackage{amsmath}
\usepackage{amssymb}
\usepackage{graphicx}
\usepackage{calc}


\makeatother

\providecommand{\casename}{Case}
\providecommand{\conjecturename}{Conjecture}
\providecommand{\definitionname}{Definition}
\providecommand{\lemmaname}{Lemma}
\providecommand{\propositionname}{Proposition}
\providecommand{\theoremname}{Theorem}

\begin{document}
\title[The CCI Conjecture for Intersection size $\le7$]{The Circuit-Cocircuit Intersection Conjecture for Intersection Size
$\le7$}
\author{Jaeho Shin}
\address{Department of Mathematical Sciences, Seoul National University, Gwanak-ro
1, Gwanak-gu, Seoul 08826, South Korea}
\email{j.shin@snu.ac.kr}
\begin{abstract}
A \emph{circuit-cocircuit intersection}, or a \emph{CCI} for short,
of a matroid is the intersection of a circuit and a cocircuit. Oxley
conjectured (1992) that a matroid with a CCI of size $k\ge4$ has
a CCI of size $k-2$. We show that the conjecture holds for $k\le7$.
\end{abstract}

\maketitle

\section{Introduction}

For a matroid $M$, a \emph{circuit-cocircuit intersection}, or a
\emph{CCI} for short, is a subset of its ground set $E(M)$ that arises
as the intersection of a circuit and a cocircuit. The collection of
CCIs is closed under dualization: a CCI of $M$ is again a CCI of
the dual matroid $M^{\ast}$. Moreover, any CCI of a minor of $M$
is again a CCI of $M$, which Oxley discovered in his paper \cite{Oxl84}
published in 1984. Later in 1992, he conjectured: 
\begin{conjecture}
\label{conj:Oxley}A matroid with a CCI of size $k\ge4$ has a CCI
of size $k-2$.
\end{conjecture}

When $k-2$ is replaced by $k-1$, the statement does not hold. For
instance, the matroid $N$ represented by the arrangement of six lines
shown in Figure \ref{fig:SixLines} serves as a counterexample with
$k=4$. Note that $N$ is a matroid whose ground set has size $6$,
which differs from a non-Fano matroid.
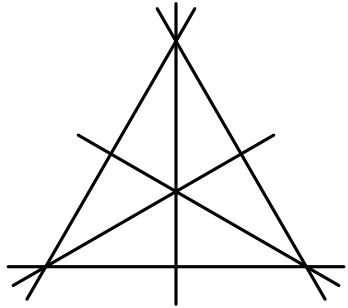
\begin{figure}[th]
\noindent \centering{}\noindent \begin{center}
\begin{tikzpicture}[font=\scriptsize]

\begin{scope}[line cap=round,xshift=0cm,yshift=0cm,rotate=0,scale=0.5]

 \foreach \x [count=\xi] in {B0,C0,A0}{
      \path (90+120*\xi:4) coordinate (\x);}

 \draw [very thick](A0)++(60:1)--++(240:8.928);
 \draw [very thick](B0)++(210:1)--++(30:8);

 \draw [very thick](A0)++(120:1)--++(-60:8.928);
 \draw [very thick](A0)++(90:1)--++(270:8);

 \draw [very thick](B0)++(180:1)--++(0:8.928);
 \draw [very thick](C0)++(-30:1)--++(150:8);
\end{scope}

\end{tikzpicture}
\par\end{center}\caption{\label{fig:SixLines}An arrangement of six lines.}
\end{figure}

In \cite{Oxl84}, it was implicitly observed that the conjecture holds
when $k=6$, but no further information was given about other $k$
values. After a long while, in 2006, Kingan and Lemos \cite{KL06}
proved the conjecture for regular matroids. However, regular matroids
are an extremely nice class of matroids, and the conjecture can be
a different problem for general matroids. This paper suggests a new
approach and shows that the conjecture holds for any CCI of size $k\le7$.\medskip{}

The following proposition leads to the concept of \emph{CCI-envelope}:
\begin{prop}[\cite{Oxl84}]
\label{prop:Oxley}Let $X$ be a CCI of a matroid $M$ with size
$k\ge4$. Then, $M$ has a minor $N$ with $\left|E(N)\right|=2k-2$
such that $r(N)=r(N^{\ast})=k-1$ and $X$ is a circuit and a cocircuit.
In particular, $E(N)\backslash X$ is independent in both $N$ and
$N^{\ast}$. 
\end{prop}

\begin{defn}
A \emph{CCI-envelope} is a matroid that contains a CCI of size $k$
and its ground set has size $2k-2$ for an integer $k\ge4$.
\end{defn}

Let $N$ be a CCI-envelope of a CCI with size $k\ge4$. Then, by Proposition
\ref{prop:Oxley}, $N$ has a minor $N'$ with $\left|E(N')\right|=2k-2=\left|E(N)\right|$,
and $N=N'$. So, a CCI-envelope is a matroid of the smallest size
that can contain a CCI, hence the name. Then, the conjecture is equivalent
to:
\begin{conjecture}
\label{conj:Oxley-1} Every CCI-envelope of a size-$k$ CCI has a
CCI of size $k-2$.
\end{conjecture}

From now on, we call Conjecture \ref{conj:Oxley-1} the \emph{circuit-cocircuit
intersection conjecture}, or the \emph{CCI conjecture} for short.\medskip{}

Let us remind readers of a well-known fact: a circuit minus a flat
cannot have size $1$, and neither can a CCI, as it is expressed as
a circuit minus a hyperplane. Switching from the (co)circuit argument
to the (co)hyperplane argument is our starting point, and we utilize
the hierarchy structure of flats\@. So, we prefer viewing a CCI as
the complement of the union of a hyperplane and a cohyperplane.

\subsubsection*{Terminological note.}

Some authors use the term ``corank'' to mean dual rank\textemdash the
rank of the dual matroid. In this paper, the corank of $A\subseteq E(M)$
for a matroid $M$ refers to the number $r(M)-r(M|A)$. Its use makes
our arguments using the flat system of matroids more convenient.

\section{\label{sec:Proof}Proof of the Main Theorem}

\subsection{Definitions and lemmas}

Let $N$ be a CCI-envelope of a CCI $X$ with size $k\ge4$. Throughout
this section, we use the letter $Y$ to denote the subset $E(N)\backslash X$
and denote by $J_{1},\dots,J_{k-2}$ all size-$(k-3)$ subsets of
$Y$.

For a subset $J\subset Y$ of size $k-3$, consider the hyperplanes
$\overline{J\cup\left\{ x\right\} }$ for all $x\in X$, then their
intersections with $X$ partition $X$ into $\left\{ X_{1},\dots,X_{m}\right\} $
for some $m\ge2$, where $J\cup X_{i}$ with $i\in\left[m\right]=\left\{ 1,\dots,m\right\} $
are all those hyperplanes. We call this partition the \emph{hyperplane-partition
of $X$ in $N$ with respect to $J$} and denote it by $\mathcal{X}=\left\{ J;X_{1},\dots,X_{m}\right\} $.
With a fixed CCI $X$, we often refer to $\mathcal{X}$ as a hyperplane-partition
\emph{of $N$}. The integer partition $\sum_{i\in\left[m\right]}\left|X_{i}\right|$
of $k$ is called the \emph{type} of $\mathcal{X}$.

For a permutation $\sigma\in\mathfrak{S}_{m}$, we may write $\mathcal{X}=(J;X_{\sigma(1)},\dots,X_{\sigma(m)})$
and call it the hyperplane-partition of \emph{type} $\left(\left|X_{\sigma(1)}\right|,\dots,\left|X_{\sigma(m)}\right|\right)$.
When using this ordered notation, we assume that $\sigma$ is chosen
such that $\left|X_{\sigma(1)}\right|\le\cdots\le\left|X_{\sigma(m)}\right|$
unless otherwise specified.

For the dual matroid $N^{\ast}$, $Y$ is a hyperplane of $N^{\ast}$,
and the hyperplane-partition of $X$ in $N^{\ast}$ is defined, which
we call the\emph{ cohyperplane-partition of $X$ in $N$}.

Let $\left(J;X_{1},\dots,X_{m}\right)$ and $\left(J';X'_{1},\dots,X'_{m'}\right)$
be (co)hyperplane-partitions. Then, the nonempty sets among $\left\{ X_{i}\cap X'_{j}:i\in\left[m\right],j\in\left[m'\right]\right\} $
form a partition of $X$, which we call the partition \emph{induced}
by those (co)hyperplane-partitions.
\begin{lem}
\label{lem:tool-1}Suppose that $\mathcal{X}=(J;X_{1},X_{2})$ and
$\mathcal{X}'=(J';X_{1}',X_{2}',\dots,X_{m}')$ for $m\ge2$ are two
hyperplane-partitions of $N$. If $\left\{ x_{1}\right\} \subset X_{1}$
and $\left\{ x_{2}\right\} \subset X_{2}$ are two size-$1$ sets
of the induced partition, the union of which is not contained in an
$X_{j}'$, then $\left(J\cap J'\right)\cup\left\{ x_{1},x_{2}\right\} $
is a hyperplane, and $X\backslash\{x_{1},x_{2}\}$ is a CCI of size
$k-2$.
\end{lem}

\begin{proof}
Without loss of generality, let $\left\{ x_{1}\right\} =X_{1}\cap X_{1}'$
and $\left\{ x_{2}\right\} =X_{2}\cap X_{2}'$. It follows that $J\cap J'$
is a flat of corank $3$, $\left(J\cap J'\right)\cup\left\{ x_{l}\right\} $
for $l=1,2$ are flats of corank $2$, and $A:=\left(J\cap J'\right)\cup\left\{ x_{1},x_{2}\right\} $
is a subset of corank $1$. Because $\left\{ x_{1},x_{2}\right\} $
is not contained in $X_{1}$ nor $X_{2}$, we have $\overline{A}\cap Y\neq J$,
while the assumption that $\left\{ x_{1},x_{2}\right\} $ is not contained
in an $X_{j}'$ forces $\overline{A}\cap Y\neq J'$. Therefore $\overline{A}\cap Y=J\cap J'$.

Suppose $x\in\overline{A}$ for some $x\in X\backslash\{x_{1},x_{2}\}$,
then either $x\in X_{1}$ or $x\in X_{2}$. If $x\in X_{1}$, then
$\left(J\cap J'\right)\cup\left\{ x_{1},x\right\} $ has corank $1$,
and $\overline{\left(J\cap J'\right)\cup\left\{ x_{1},x\right\} }=\overline{A}$.
Moreover, $\overline{\left(J\cap J'\right)\cup\left\{ x_{1},x\right\} }=J\cup X_{1}$
since $J\cup X_{1}$, which contains $\left(J\cap J'\right)\cup\left\{ x_{1},x\right\} $,
is a hyperplane. So, we have $\overline{A}=J\cup X_{1}$, a contradiction.
Therefore $x\notin X_{1}$, and $x\notin X_{2}$ in the same way,
and $x\notin X=X_{1}\cup X_{2}$, a contradiction. This argument proves
$\overline{A}=A$ and that $X\backslash\{x_{1},x_{2}\}=E(N)\backslash(Y\cup A)$
is a CCI of size $k-2$.
\end{proof}
\begin{lem}
\label{lem:tool-2}Let $\mathcal{X}=\left(J;X_{1},X_{2}\right)$ and
$\mathcal{X}'=\left(J';X'_{1},X'_{2}\right)$ with $J\neq J'$ be
hyperplane-partitions of $N$, then $\left\{ X_{1},X_{2}\right\} \neq\left\{ X_{1}',X_{2}'\right\} $.
So, for the partition $\mathcal{Z}=\left\{ Z_{i}:i\in\left[m\right]\right\} $
of $X$ induced by $\mathcal{X}$ and $\mathcal{X}'$, one has $m\in\{3,4\}$.
If $Z_{i_{1}}$ and $Z_{i_{2}}$ are two sets of $\mathcal{Z}$ such
that $Z_{i_{1}}\cup Z_{i_{2}}$ is not contained in a set of $\mathcal{X}$
or $\mathcal{X}'$, then $(J\cap J')\cup Z_{i_{1}}\cup Z_{i_{2}}$
is a hyperplane. Moreover, when $m=3$, $\left|Z_{i}\right|>1$ for
all $i\in\left[3\right]$.
\end{lem}

\begin{proof}
Suppose $\left\{ X_{1},X_{2}\right\} =\left\{ X_{1}',X_{2}'\right\} $,
and let $X_{1}=X_{1}'$ and $X_{2}=X_{2}'$ without loss of generality.
Then, $(J\cap J')\cup X_{1}=(J\cup X_{1})\cap(J'\cup X_{1})$ is a
flat of corank $2$, and so is $(J\cap J')\cup X_{2}$. Therefore
$(J\cap J')\cup X_{1}\cup X_{2}$ is contained in a hyperplane, which
is impossible because $r_{N}(X_{1}\cup X_{2})=r_{N}(X)=k-1$. Thus,
$\left\{ X_{1},X_{2}\right\} \neq\left\{ X_{1}',X_{2}'\right\} $.

For the partition $\left\{ Z_{i}:i\in\left[m\right]\right\} $, it
follows $m\in\{3,4\}$. Let $m=3$. If $\left|Z_{1}\right|=1$, then
$(J\cap J')\cup Z_{2}\cup Z_{3}$ is contained in a hyperplane, which
is also impossible because $r_{N}(Z_{2}\cup Z_{3})=k-1$. Thus, $\left|Z_{1}\right|>1$.
In the same way, $\left|Z_{2}\right|>1$ and $\left|Z_{3}\right|>1$.

Now, $(J\cap J')\cup Z_{i_{1}}\cup Z_{i_{2}}$ has rank $k-1$, and
adding an element from $X\backslash(Z_{i_{1}}\cup Z_{i_{2}})$ or
$Y\backslash(J\cap J')$ to it increases the rank since $Z_{i_{1}}\cup Z_{i_{2}}$
is not contained in a set of $\mathcal{X}$ or $\mathcal{X}'$. Thus,
$(J\cap J')\cup Z_{i_{1}}\cup Z_{i_{2}}$ is a hyperplane.
\end{proof}
\begin{lem}
\label{lem:1and1,1}Let $(J;\{x\},\dots)$ and $(J^{\ast};\{x^{\ast}\},\dots)$
with $J\neq J^{\ast}$ be a hyperplane-partition and a cohyperplane-partition
of $N$, respectively. If $x\neq x^{\ast}$ or one of the partitions
has type $(1,1,\dots)$, then $N$ satisfies the CCI conjecture.
\end{lem}

\begin{proof}
If $x\neq x^{\ast}$, then $J\cup\{x\}$ and $J^{\ast}\cup\{x^{\ast}\}$
are a hyperplane and a cohyperplane, respectively, and $X\backslash\{x,x^{\ast}\}$
is a CCI of size $k-2$. If $x=x^{\ast}$, then one of the given partitions
has type $(1,1,\dots)$, and suppose $(J^{\ast};\{x^{\ast}\},\dots)$
has type $(1,1,\dots)$ without loss of generality. Then, there exists
a size-$1$ set $\{x^{\ast\ast}\}$ of $(J^{\ast};\{x^{\ast}\},\dots)$
with $x^{\ast\ast}\neq x^{\ast}$, and $X\backslash\{x,x^{\ast\ast}\}$
is a CCI of size $k-2$.
\end{proof}
\begin{lem}
\label{lem:tool-4}Let $\left(J;X_{1},X_{2}\right)$ be a hyperplane-partition.
If $\left(J^{\ast};X_{0}^{\ast},\dots\right)$ with $J^{\ast}\neq J$
is a cohyperplane-partition, then $\left|X_{1}\cap X_{0}^{\ast}\right|\neq\left|X_{1}\right|-1$
and $\left|X_{2}\cap X_{0}^{\ast}\right|\neq\left|X_{2}\right|-1$.
In particular, if $\left(J;X_{1},X_{2}\right)$ has  type $\left(2,k-2\right)$
or $\left(3,k-3\right)$, and $\left(J^{\ast};X_{1}^{\ast},X_{2}^{\ast}\right)$
is a cohyperplane-partition, then $X_{1}\subseteq X_{1}^{\ast}$ or
$X_{1}\subseteq X_{2}^{\ast}$.
\end{lem}

\begin{proof}
If $\left|X_{0}^{\ast}\cap X_{1}\right|=\left|X_{1}\right|-1$, then
$\left|X_{0}^{\ast}\cup X_{2}\right|=\left|X\right|-1$, and $\left(E(N)\backslash X\right)\cup X_{0}^{\ast}\cup X_{2}$
is a union of a hyperplane and a cohyperplane with size $2k-3$, a
contradiction. So, $\left|X_{0}^{\ast}\cap X_{1}\right|\neq\left|X_{1}\right|-1$,
and similarly, $\left|X_{0}^{\ast}\cap X_{2}\right|\neq\left|X_{2}\right|-1$.

In particular, suppose that $\left(J;X_{1},X_{2}\right)$ has  type
$\left(2,k-2\right)$ or $\left(3,k-3\right)$ and that $\left(J^{\ast};X_{1}^{\ast},X_{2}^{\ast}\right)$
is a cohyperplane-partition. If $X_{1}\nsubseteq X_{1}^{\ast}$ and
$X_{1}\nsubseteq X_{2}^{\ast}$, then $X_{1}^{\ast}\cap X_{1}\neq\emptyset\neq X_{2}^{\ast}\cap X_{1}$,
and either $\left|X_{1}^{\ast}\cap X_{1}\right|=1$ or $\left|X_{2}^{\ast}\cap X_{1}\right|=1$
since $\left|X_{1}\right|=2$ or $3$; hence, either $\left|X_{2}^{\ast}\cap X_{1}\right|=\left|X_{1}\right|-1$
or $\left|X_{1}^{\ast}\cap X_{1}\right|=\left|X_{1}\right|-1$, contrary
to the above argument. Thus we conclude $X_{1}\subseteq X_{1}^{\ast}$
or $X_{1}\subseteq X_{2}^{\ast}$.
\end{proof}

\subsection{Proof of the main theorem}

We prove the CCI conjecture for $k=7$, and the remaining cases $4\le k\le6$
can be handled similarly. Let $N$ be a CCI-envelope of a CCI $X$
with size $7$. If $N$ has a (co)hyperplane-partition $\left(J;-\right)$
with a size-$2$ set $X_{0}$, then $X\backslash X_{0}=E(N)\backslash(Y\cup(J\cup X_{0}))$
is a CCI of size $5$, where $Y=E(N)\backslash X$ is a hyperplane
and a cohyperplane. So, throughout this subsection, for the sake of
convenience, we may assume that no set of a (co)hyperplane-partition
has size $2$. Then, those integer partitions of $7$ that (co)hyperplane-partitions
can induce are:
\[
4^{1}3^{1},\,3^{2}1^{1},\,4^{1}1^{3},\,3^{1}1^{4},\,\text{ and }\,1^{7}.
\]
For $i\in[5]$, let $\mathcal{X}_{i}=(J_{i};-)$ and $\mathcal{X}_{i}^{\ast}=(J_{i};-)$
be hyperplane- and cohyperplane-partitions, respectively.
\begin{prop}
\label{prop:34vs34}Let $\mathcal{X}_{1}$ be a type-$4^{1}3^{1}$
hyperplane-partition of $N$. If some $\mathcal{X}_{i}^{\ast}$ for
$i\in[5]\backslash\{1\}$ has the same type, then $\mathcal{X}_{1}$
and $\mathcal{X}_{i}^{\ast}$ are the same as partitions of $X$.
\end{prop}

\begin{proof}
Let $\mathcal{X}_{1}=(J_{1};X_{1},X_{1}')$ be the given hyperplane-partition
of type $(3,4)$, and suppose $\mathcal{X}_{i}^{\ast}=(J_{i};X_{i},X_{i}')$
for some $i\in[5]\backslash\{1\}$ has the same type. Then, either
$X_{1}=X_{i}$ or $X_{1}\subset X_{i}'$ and $X_{i}\subset X_{1}'$
by Lemma \ref{lem:tool-4} since $\left|X_{1}\right|=\left|X_{i}\right|=3$.
But, if $X_{1}\subset X_{i}'$, then $\left\{ X_{1},X_{i},X_{i}'\backslash X_{1}\right\} $
with $\left|X_{i}'\backslash X_{1}\right|=1$ is the partition induced
by $\mathcal{X}_{1}$ and $\mathcal{X}_{i}^{\ast}$, contrary to Lemma
\ref{lem:tool-2}. Thus $X_{1}=X_{i}$, and $\left\{ X_{1},X_{1}'\right\} =\left\{ X_{i},X_{i}'\right\} $.
\end{proof}
Thus, we have two cases for the hyperplane- and cohyperplane-partitions
of $N$:
\begin{enumerate}
\item All hyperplane-partitions or all cohyperplane-partitions have type
$4^{1}3^{1}$. 
\item There is a two-element subset $\{i,j\}\subset[5]$ such that for any
$l\in[5]\backslash\{i,j\}$, both $\mathcal{X}_{l}$ and $\mathcal{X}_{l}^{\ast}$
have types other than $4^{1}3^{1}$, and hence their types contain
a summand of $1$.
\end{enumerate}
First, we resolve case (1) (Proposition \ref{prop:all34}) and then
proceed to settle case (2), and hence obtain the main theorem (Theorem
\ref{thm:main}).
\begin{prop}
\label{prop:all34}If either all hyperplane-partitions or all cohyperplane-partitions
have type $4^{1}3^{1}$, then $N$ satisfies the CCI conjecture.
\end{prop}

\begin{proof}
Let $\mathcal{X}_{i}=(J_{i};X_{i},X_{i}')$ be a type-$(3,4)$ hyperplane-partition,
and $\mathcal{X}_{i}^{\ast}=(J_{i};-)$ a cohyperplane-partition.
Then, $\mathcal{X}_{i}^{\ast}$ cannot have type $4^{1}3^{1}$ by
Proposition \ref{prop:34vs34} and Lemma \ref{lem:tool-2}. If $\mathcal{X}_{i}^{\ast}$
has type $4^{1}1^{3}$, then for the size-$4$ set $Z$ of it, $X_{i}\nsubseteq Z$
by Lemma \ref{lem:tool-2}, and there is a size-$1$ set $Z'$ of
it that is contained in $X_{i}$. Hence, $E(N)\backslash(J_{i}\cup X_{i})$
is a size-$5$ CCI. If $\mathcal{X}_{i}^{\ast}$ has type $3^{1}1^{4}$,
then for the size-$3$ set $Z$ of it, we have $Z\neq X_{i}$, and
there is a size-$1$ set $Z'$ that is contained in $X_{i}$. So,
$E(N)\backslash(J_{i}\cup X_{i})$ is a size-$5$ CCI. If $\mathcal{X}_{i}^{\ast}$
has type $1^{7}$, we see that $N$ has a size-$5$ CCI in the same
manner. Thus we may assume all cohyperplane-partitions have type $3^{2}1^{1}$.

For two hyperplane-partitions $\mathcal{X}_{i}=(J_{i};X_{i},X_{i}')$
and $\mathcal{X}_{j}=(J_{j};X_{j},X_{j}')$, we know $X_{i}\neq X_{j}$,
$X_{i}\nsubseteq X_{j}'$ and $X_{j}\nsubseteq X_{i}'$ by Lemma \ref{lem:tool-2}.
If $\left|X_{i}\cap X_{j}\right|=2$, then $\left|X_{i}\backslash X_{j}\right|=\left|X_{j}\backslash X_{i}\right|=1$
and by Lemma \ref{lem:tool-1}, $X\backslash(X_{i}\triangle X_{j})$
is a size-$5$ CCI, where $\triangle$ denotes the symmetric difference.
Thus we may assume $\left|X_{i}\cap X_{j}\right|=1$ for all $i\neq j$.

Let $\mathcal{X}_{i}^{\ast}=(J_{i};\{x_{0}\},\dots)$ be a type-$(1,3,3)$
cohyperplane-partition. If $x_{0}\in X_{i}$, then $E(N)\backslash(J_{i}\cup X_{i})$
is a size-$5$ CCI. If $x_{0}\notin X_{i}$, then there exists another
$X_{j}$ that does not contain $x_{0}$, where $x_{0}\in X\backslash(X_{i}\cup X_{j})$.
Then, $(J_{i}\cap J_{j})\cup(X\backslash(X_{i}\triangle X_{j}))$
is a hyperplane with $\left|X\backslash(X_{i}\triangle X_{j})\right|=3$
by Lemma \ref{lem:tool-2}, while $J_{i}\cup\{x_{0}\}$ is a cohyperplane.
Hence, $E(N)\backslash(J_{i}\cup(X\backslash(X_{i}\triangle X_{j})))$
is a size-$5$ CCI.

By symmetry, the above argument remains valid when the terms ``hyperplane(-)''
and ``cohyperplane(-)'' are interchanged, and thus, the proof is
complete.
\end{proof}
\begin{thm}
\label{thm:main}The CCI conjecture holds for any CCI of size $k=7$.
\end{thm}

\begin{proof}
By Proposition \ref{prop:all34}, without loss of generality, we may
assume that for any $i\in[3]$, the types of $\mathcal{X}_{i}$ and
$\mathcal{X}_{i}^{\ast}$ contain a summand of $1$. Then, we have
two cases:
\begin{casenv}
\item There exists a hyperplane-partition $\mathcal{X}_{i}=(J_{i};\{x_{i}\},\dots)$
and a cohyperplane-partition $\mathcal{X}_{j}^{\ast}=(J_{j};\{x_{j}\},\dots)$
with $x_{i}\neq x_{j}$ for distinct $i$ and $j$ in $\left\{ 1,2,3\right\} $.
In this case, by Lemma \ref{lem:1and1,1}, $N$ satisfies the CCI
conjecture.
\item All $\mathcal{X}_{i}$ and $\mathcal{X}_{i}^{\ast}$ for $i\in\left[3\right]$
have type $3^{2}1^{1}$ and their size-$1$ sets coincide. In this
case, $\mathcal{X}_{1}=(J_{1};\{x_{0}\},X_{1},X_{1}')$ and $\mathcal{X}_{2}=(J_{2};\{x_{0}\},X_{2},X_{2}')$
differ as partitions of $X$ by Lemma \ref{lem:tool-2}. Then, $\left|X_{1}\cap X_{2}\right|=\left|X_{1}'\cap X_{2}'\right|\neq0$
and $\left|X_{1}\cap X_{2}'\right|=\left|X_{1}'\cap X_{2}\right|\neq0$,
and these two numbers sum up to $3$. Let $\left|X_{1}\cap X_{2}\right|=2$
and $\left|X_{1}\cap X_{2}'\right|=1$ without loss of generality.
Using Lemma \ref{lem:tool-2}, we see that either $(J_{1}\cap J_{2})\cup(X_{1}\triangle X_{2})$
or $(J_{1}\cap J_{2})\cup(X_{1}\triangle X_{2})\cup\{x_{0}\}$ is
a hyperplane, while $J_{1}\cup\{x_{0}\}$ is a cohyperplane. Therefore,
$E(N)\backslash(J_{1}\cup(X_{1}\triangle X_{2})\cup\{x_{0}\})$ is
a size-$5$ CCI.
\end{casenv}
Thus, $N$ satisfies the CCI conjecture, and the proof is complete.
\end{proof}

\end{document}